\documentclass[11pt,reqno]{amsart}
\usepackage{amssymb}
\usepackage{enumerate}
\usepackage[pagewise]{lineno}
\usepackage[all]{xy}
\usepackage{url}
\usepackage{cite}

\numberwithin{equation}{section}

\newtheorem{thm}{Theorem}[section]
\newtheorem{cor}[thm]{Corollary}
\newtheorem{lem}[thm]{Lemma}
\newtheorem{prop}[thm]{Proposition}

\theoremstyle{definition}
\newtheorem{rem}[thm]{Remark}
\newtheorem{note}[thm]{Notation}

\newtheorem{defi}[thm]{Definition}

\DeclareMathOperator{\NL}{\mathrm{NL}}

\DeclareMathOperator{\Ima}{\mathrm{Im}}
\DeclareMathOperator{\p3}{\mathbb{P}^3}

\DeclareMathOperator{\pr}{\mathrm{pr}}
\DeclareMathOperator{\Spec}{\mathrm{Spec}}
\DeclareMathOperator{\red}{\mathrm{red}}
\DeclareMathOperator{\N}{\mathcal{N}}
\DeclareMathOperator{\T}{\mathcal{T}}
\DeclareMathOperator{\I}{\mathcal{I}}
\DeclareMathOperator{\mo}{\mathcal{O}}

\DeclareMathOperator{\Hc}{\mc{H}om}

\newcommand{\mb}[1]{\mathbb{#1}}
\newcommand{\mc}[1]{\mathcal{#1}}
\newcommand{\mr}[1]{\mathrm{#1}}
\newcommand{\ov}[1]{\overline{#1}}
\newcommand{\mf}[1]{\mathfrak{#1}}
\newcommand{\mbf}[1]{\mathbf{#1}}

\begin{document}

\title[Local topological obstruction for divisors]{Local topological obstruction for
divisors}

\author[I. Biswas]{Indranil Biswas}

\address{School of Mathematics, Tata Institute of Fundamental
Research, Homi Bhabha Road, Mumbai 400005, India}

\email{indranil@math.tifr.res.in}

\author[A. Dan]{Ananyo Dan}

\address{BCAM -- Basque Centre for Applied Mathematics, Alameda de Mazarredo 14,
48009 Bilbao, Spain} 

\email{adan@bcamath.org}



\begin{abstract}
Given a smooth, projective variety $X$ and
an effective divisor $D\,\subseteq\, X$, it is well-known that 
the (topological) obstruction 
to the deformation of the fundamental class of $D$ as a Hodge class,
lies in $H^2(\mo_X)$. In this article, we replace $H^2(\mo_X)$ by 
$H^2_D(\mo_X)$ and give an analogous topological obstruction theory.
We compare the resulting local topological obstruction theory
with the geometric obstruction theory (i.e., the obstruction to the
deformation of $D$ as an effective Cartier divisor of a first order 
infinitesimal deformations of $X$).
 We apply this to study the jumping locus of 
families of linear systems and the Noether-Lefschetz locus. Finally,
we give examples of first order deformations $X_t$ of $X$ for which 
the cohomology class $[D]$
deforms as a Hodge class but $D$ \emph{does not} lift as an 
effective Cartier divisor of $X_t$.
\end{abstract}

\maketitle

\tableofcontents

\section{Introduction}

The base field $k$ is always assumed to be algebraically closed of characteristic zero. Consider a family of smooth,
projective varieties \[\pi\, :\, \mc{X}\,\longrightarrow\, B\]
parameterized by a complex manifold $B$. Fix a closed point 
$o \in B$ and denote by $X\,:=\,\pi^{-1}(o)$ the fiber over $o$. 
Let $D \,\subseteq\, X$ be an effective divisor.
Recall that, by Hodge decomposition, $H^2(X,\,\mb{C})$ decomposes
as a direct sum of sub-vector spaces $H^{i,2-i}(X)$ for 
$0 \,\le\, i \,\le\, 2$. The integral elements of $H^{1,1}(X)$, meaning the elements
of $H^{1,1}(X) \cap (H^2(X,\,\mb{Z})/(\mr{torsion}))$, are known as \emph{Hodge classes}.
It is well-known that the fundamental class of a divisor, given by the 
first Chern class of the associated line bundle, is a Hodge class, which is
proved using the exponential short exact sequence.

There is an obvious ``topological'' obstruction for
an effective divisor $D \,\subseteq\, X$ to deform along with the variety $X$: the 
fundamental class $[D]$ of the divisor $D$ is needed to deform as a 
Hodge class. Using the Lefschetz $(1,1)$--theorem, this topological
obstruction is precisely the obstruction to the deformation of the
invertible sheaf $\mo_X(D)$. Using the exponential short exact
sequence, one can then check that this obstruction actually lies 
in $H^2(X,\, \mo_X)$. We will say that \emph{the fundamental class of $D$
deforms as a Hodge class along $t \in T_oB$} 
if the invertible sheaf $\mo_X(D)$ 
deforms to an invertible sheaf on the first order 
infinitesimal deformation $X_t$
of $X$, in the direction of $t$.
If the obstruction vanishes, we may further ask
whether the divisor $D$ deforms as an effective Cartier divisor in $X_t$?
This happens if and only if
the section $s_D$ of $\mo_X(D)$
cutting out the divisor $D$ deforms to a global section 
of the infinitesimal deformations of the 
invertible sheaf $\mo_X(D)$. It is known that the obstruction to
such deformations of the section $s_D$ lies in 
$H^1(X,\, \mo_X(D))$ (see \cite[\S~$3.3$]{S1}). In this article, 
we replace the classical topological obstruction space 
$H^2(X,\, \mo_X)$ by the local cohomology group $H^2_D(\mo_X)$.

We prove the following (see Theorem \ref{brill04}):

\begin{thm}\label{th:br01}
Consider deformations of $X$ parameterized by $B$ with base point
$o\,\in\, B$, as above. To each effective divisor 
$D\, \subseteq\, X$ and tangent vector $t \,\in\, T_oB$, there is a
\emph{canonically} associated element $\xi_D^t \,\in\, H^2_D(\mo_X)$ 
that satisfies the following condition: The vanishing
of $\xi_D^t$ is equivalent to the existence of an integer $n\,>\,0$ such that
$nD$ deforms as an effective Cartier divisor, as $X$ deforms in the direction of $t$.
\end{thm}

Theorem \ref{th:br01} says that given an effective
divisor $D$ in $X$, the obstruction to the deformation of 
a high enough multiple of $D$ as an effective Cartier divisor
lies in $H^2_D(\mo_X)$. Of course, if the 
obstruction vanishes then the fundamental class of $D$ 
deforms as a Hodge class. The question then arises, if no 
multiple of $D$ deforms as an effective Cartier divisor, is it 
still possible for the fundamental class of $D$ to deform as 
a Hodge class? As noted in Remark \ref{rem:ref}, this is possible.
However, it can be shown that the difference between the two obstruction 
theories (i.e., the topological obstruction for the deformation of the 
fundamental class of $D$ as a Hodge class
and the geometric obstruction for the 
deformation of a multiple of $D$ as an effective Cartier divisor)
arises from $1$-cocycles of the structure sheaf of an open subset $U$ of 
$X$ such that $X \backslash U$ is of codimension at least $2$ in $X$.
In particular, we prove (see Theorem \ref{th:br03}):

\begin{thm}\label{thintro1}
 There exists a proper, closed subvariety $T \,\subseteq\, D$ (proper meaning 
 the dimension of $T$ is strictly less than that of $D$)
 such that for the natural morphism 
 \[\Phi_{(T \subseteq D)}\,:\, H^1(\mo_{X\backslash T}) \,\longrightarrow\, H^1(\mo_{X\backslash
D}) \,\longrightarrow\, H^2_D(\mo_X)\, , \]
 and any tangent vector $t \in T_oB$, the fundamental class $[D]$ deforms 
 as a Hodge class in the direction of $t$ if and only if the 
 image of $\xi^t_D \in H^2_D(\mo_X)$ defined
 in Theorem \ref{th:br01} above,
 maps to zero in $\mr{coker}(\Phi_{(T \subseteq D)})$.
\end{thm}

In other words, the topological obstruction space is in fact 
contained in the cokernel of the morphism $\Phi_{(T \subseteq D)}$ 
mentioned above. The natural question to ask is, whether Theorem 
\ref{thintro1} holds for any proper, closed subvariety of the divisor $D$?
Before we answer this question, we fix a couple of definitions. 
Given a proper, closed subvariety $T$ of $D$, we say that $D$ is 
$T$-\emph{semi-regular} if Theorem \ref{thintro1} above
holds for such a choice of $T$. We say that $D$ is \emph{effective
Lefschetz} if it satisfies the following property: for any $t \,\in\, T_oB$,
the fundamental class $[D]$ deforms as a Hodge class along $t$ if and only if
a multiple of $D$ deforms as an effective Cartier divisor, as $X$ deforms 
along $t$. In this context the following is proved (see Theorem \ref{brnew01}):

\begin{thm}
 The divisor $D$ is $T$-semi-regular for every proper, closed 
 subvariety $T$ of $D$ if and only if $D$ is effective Lefschetz.
\end{thm}

We give an explicit example in Theorem \ref{brill22} of a divisor 
which is not effective Lefschetz. See Remark \ref{rem:ref} for a 
simpler example, given by the referee.

Let us now compare the above results with some classical results.
We will call an effective divisor $D \,\subseteq\, X$ 
 \emph{saturated} if for every first order infinitesimal deformation of $X$,
the divisor $D$ deforms as an effective Cartier divisor 
if and only if every positive 
multiple of $D$ deforms as an effective Cartier divisor.
By Theorem \ref{th:br01}, the divisor $D$ is saturated if and only if the natural 
morphism from the geometric obstruction space 
$\mbf{O}_D^{\pi} \,\subseteq\, H^1(\N_{D|X})$ (i.e.,
obstruction to the deformation of $D$ as an effective Cartier divisor)
to $H^2_D(\mo_X)$, induced by the natural morphism from $H^1(N_{D|X})$ to 
$H^2_D(\mo_X)$ (see \S~\ref{sec:cup}), is injective (see Corollary 
\ref{cor:sat}).
Note that, for saturated divisors, being effective Lefschetz is equivalent
to the injectivity of the natural morphism from the geometric
obstruction space $\mbf{O}^\pi_D$ to the topological obstruction space $\mr{Obs}^{\pi}_D\,\subseteq\, H^2(\mo_X)$
(i.e., the obstruction to deforming the fundamental class of $D$ as 
a Hodge class).
Note that semi-regular sub-varieties obey such an 
injectivity statement. Recall that the theory of semi-regular 
sub-varieties was first developed by Kodaira--Spencer \cite{kod1}
in the case of divisors, and subsequently
it was generalized for the higher codimensions by Bloch \cite{b1}.
However, semi-regularity is a rather rare phenomenon (see \cite{ink}).
In contrast, saturated divisors
are much more common than semi-regular divisors. For example,
a high enough multiple of any effective divisor is saturated.

{\bf{Application:}} One can immediately apply these results 
to study the jumping locus of families of 
linear systems. Consider a flat family of smooth, projective varieties $\mc{X}$ parameterized 
by a complex manifold $B$. Fix a point $o \in B$. Let $\mc{L}$ be an invertible sheaf over 
$\mc{X}$. Denote by $W$ the \emph{universal sub-locus} of $B$ containing the point $o$
such that for every $w \,\in\, W$, the dimension of
the linear system associated to
$\mc{L}_w\,:=\,\mc{L}|_{\mc{X}_w}$ is constant
and every sub-loci satisfying the same properties factor
through $W$ (see \S~\ref{rev2}
for a precise definition). We prove (see Theorem \ref{brill21}):

\begin{thm}
Denote by $\mc{L}_o$ the restriction of $\mc{L}$ to the fiber over $o$.
The tangent space at $o$ to $W$ coincides with that of the parameterizing space $B$ (i.e., up to
first order, the dimension of the linear system 
over the point $o \in B$ does not jump) if and 
only if every effective divisor $D$, in the linear system defined by $\mc{L}_o$,
is $T$-semi-regular for every proper, closed subvariety $T \,\subseteq\, D$.
\end{thm}

\subsection*{List of frequently used notations}

We list a set of frequently used notations. Let $X$ be a smooth, 
projective variety and $D$ be an effective divisor in $X$. We then 
denote by:

\begin{enumerate}
\item $\mc{H}^i_D(-)$ the local cohomology sheaf defined in \S~ 
\ref{sec:supp}.
 \item $\lrcorner \{D\}$ the contraction map from the tangent sheaf 
 $\T_X$ to $\mc{H}^1_D(\mo_X)$, defined in \S~\ref{inn}.
 \item $\lrcorner \{D\}'$ the contraction map from the normal sheaf 
 $\N_{D|X}$ to $\mc{H}^1_D(\mo_X)$, defined in \S~\ref{inn}.
 \item $\pi_D$ is the morphism from $H^1(\T_X)$ to 
 $H^1(\mc{H}^1_D(\mo_X)) \cong H^2_D(\mo_X)$ induced by $\lrcorner \{D\}$. 
 \item $\pi'_D$ is the morphism from $H^1(\N_{D|X})$ to 
 $H^1(\mc{H}^1_D(\mo_X)) \cong H^2_D(\mo_X)$ induced by $\lrcorner \{D\}'$.
 \item $\Phi_D$ is the natural morphism from $H^2_D(\mo_X)$
 to $H^2(\mo_X)$ guaranteed by Lemma \ref{ph14} below.
 \item $\rho_{_D}$ is the natural morphism from $\T_X$ to $\N_{D|X}$
 defined in \S~\ref{sec:top}. 
 \item $\mr{Ob}_D$ from $H^1(\T_X)$ to $H^1(\N_{D|X})$ is 
 the natural obstruction map, induced by $\rho_{_D}$.
 \item $\mr{KS}$ denotes the Kodaira-Spencer map (see \S~\ref{sec:hod}).
 \item $\Phi_{(T \subseteq D)}$ for a proper closed subvariety $T$ of $D$ is 
 the composition \[H^1(\mo_{X\backslash T}) \to H^1(\mo_{X \backslash D})
 \to H^2_D(\mo_X).\]
 \item $\eta_T$ is the natural morphism from the cokernel of 
 $\Phi_{(T \subseteq D)}$ to $H^2(\mo_X)$ defined in \eqref{eq:eta}.
\end{enumerate}

\section{Preliminaries}\label{rev3}

In this section we recall the local description of an effective divisor $D$
contained in a smooth projective variety $X$. Using this, we define 
inner multiplications $\lrcorner \{D\}$ and $\lrcorner \{D\}'$. We end
the section with recollections of some 
basics on Hodge loci.

\subsection{Cohomology with support}\label{sec:supp}

Let $X$ be a smooth projective variety and $D \,\subseteq\, X$
be an effective divisor.
Let $\mc{F}$ be a sheaf of $\mo_X$-modules. 
A (local) section $s$ of $\mc{F}$ is said to be \emph{supported
over $D$} if $s$ is annihilated by a power of the ideal sheaf 
$\I_{D|X}$ of $D$ (as a subscheme of $X$). In particular, 
\emph{the sheaf $\underline{\Gamma}_D(\mc{F})$ of sections of $\mc{F}$
supported over $D$} is the sheaf associated to the presheaf which 
assigns to an open subset $U \,\subseteq\, X$, the $\mo_X(U)$-module
\[\{s \,\in\, \mc{F}(U)\,\mid\,\I_{D|X}(U)^n.s\,=\,0\ \ \text{ for some }\, n\,>\,0\}\, .\]
Note that $\underline{\Gamma}_D(\mc{F})$ is isomorphic to 
$\underline{\Gamma}_{D_{\red}}(\mc{F})$, where $D_{\red}$ is the reduced
scheme associated to $D$. Let $\Gamma_D(\mc{F})$ be the space of 
global sections of the sheaf $\underline{\Gamma}_D(\mc{F})$.
Denote by $\mc{H}^i_D(\mc{F})$ (respectively, $H^i_D(X,\mc{F})$) the 
right derived functor of the left exact 
functor $\underline{\Gamma}_D(-)$ 
(respectively, $\Gamma_D(-)$) applied to the sheaf $\mc{F}$.
See \cite[Chapter $1$]{brod} or \cite[\S~$1$]{grh} for further details.
We now recall the following useful result that encodes the difference 
between the cohomology groups $H^i(\mc{F})$ and $H^i_D(\mc{F})$:

\begin{lem}[{\cite[Corollary $1.1.9$]{grh}}]\label{ph14}
Let $\mc{F}$ be a quasi-coherent sheaf on $X$. Let $U\,:=\, X \backslash D$ be the complement
with $j\,:\,U \,\hookrightarrow\, X$ the inclusion. There is a long exact sequence
\[
0 \,\longrightarrow\, H^0_D(X,\,\mc{F})\,\longrightarrow\, H^0(X,\,\mc{F})\,\longrightarrow\,
 H^0(U,\,\mc{F}|_U) \,\longrightarrow\, H^1_D(X,\,\mc{F})
\]
\[
\longrightarrow\, H^1(X,\,\mc{F})\,\longrightarrow\, H^1(U,\,\mc{F}|_U)
\,\longrightarrow\,H^2_D(X,\mc{F})\,\longrightarrow\, \cdots \, .\]
Similarly, there is a short exact sequence of sheaves
 \[0 \,\longrightarrow\, \mc{H}^0_D(X,\,\mc{F})
\,\longrightarrow\, \mc{H}^0(X,\,\mc{F})\,\longrightarrow\, \mc{H}^0(U,\,\mc{F}|_U)
\, {\longrightarrow}\, \mc{H}^1_D(X,\,\mc{F}) \,\longrightarrow\, 0\, ,\]
and $\mc{H}^{i+1}_D(\mc{F}) \,\cong\, R^i j_*(\mc{F}|_U)$ for all $i\,>\,0$.
\end{lem}

\subsection{Local description of cohomology class of divisors}\label{sec:loc}

Let $D \,\subseteq\, X$ be as in \S~\ref{sec:supp}. Define $U\,:=\,X \backslash
D$, and let $j\,:\,U\, \hookrightarrow\, X$ be the natural inclusion. 
Using Lemma \ref{ph14}, we have the short exact sequence
\begin{equation}\label{brill03}
0\,\longrightarrow\, \Omega^1_X\,\longrightarrow\, j_*\Omega^1_{X\backslash D}
\,\stackrel{\delta}{\longrightarrow} \,\mc{H}^1_D(\Omega^1_X) \,\longrightarrow\, 0\, .
\end{equation}
Let $\{U_i\}_{i \in I}$ be an open affine covering of $X$ such that $D \, \cap \, U_i$ 
is defined by a single equation, say $f_i\,=\,0$ with $f_i \,\in\, \Gamma(U_i,\,\mo_X)$.
For $i \not= j$, since $f_i=\lambda_{ij}f_j$ over $U_i \cap U_j$
for some invertible 
regular section $\lambda_{ij}$ over $U_i \cap U_j$, 
we have by Leibniz rule,
\[\delta\left(\frac{df_i}{f_i}\right)\, =\, 
\delta\left(\frac{df_j}{f_j}\right)+\delta\left(\frac{d\lambda_{ij}}{\lambda_{ij}}\right)\]
over $U_i\cap U_j$.
As $d\lambda_{ij}/\lambda_{ij} \in \Gamma(U_i \cap U_j, \Omega^1_X)$, \eqref{brill03} implies that the last term of the above equality 
vanishes. Hence, the collection of sections 
\[\left\{ \delta\left(\frac{df_i}{f_i}\right)\right\}_{i \in I},\]
glue to give a global section, say
$\{D\}\,\in\, H^0(\mc{H}^1_D(\Omega^1_X))$.
We now observe that in the case $\mc{F}$ is locally-free, then the 
cohomology group $H^i_D(\mc{F})$ is isomorphic to $H^{i-1}(\mc{H}^1_D(\mc{F}))$. In fact, we recall a more general result.

\begin{prop}\label{ph15}
Let $X$ be a scheme, and $Z$ be a local complete intersection 
subscheme in $X$. Let $\mc{F}$ be a sheaf of abelian groups on $X$.
Then the spectral sequence with terms $$E_2^{p,q}\,=\,H^p(X,\,\mc{H}_Z^q(X,\,
\mc{F}))$$ converges to $H_Z^{p+q}(X,\,\mc{F})$. Furthermore, if $\mc{F}$ is a locally
free $\mo_X$--module, then \[H^{p+q}_Z(X,\,\mc{F})\,\cong\, H^p(X,\,\mc{H}^q_Z(X,\,
\mc{F})),\] where $q$ is the codimension of $Z$ in $X$ and $p \,\ge\, 0$.
\end{prop}

\begin{proof}
The first statement is proven in \cite[Proposition $1.4$]{grh}.

If $\mc{F}$ is locally free, then \cite[Theorem $3.5.7$]{brun1} implies that $\mc{H}^k_Z(X,\,\mc{F})
\,=\,0$ for $k \,\not=\, q$. Since 
$$E_2^{p,q}\,=\,H^p(X,\,\mc{H}_Z^q(X,\,\mc{F}))\,\Rightarrow\, H^{p+q}_Z(X,\,\mc{F})\, ,$$
we conclude that $H^{p+q}_Z(X,\,\mc{F}) \,\cong\, H^p(X,\,\mc{H}^q_Z(\mc{F}))$.
This completes the proof.
\end{proof}

In particular, Proposition \ref{ph15} implies that $H^0(\mc{H}^1_D(\Omega^1_X))
\,\cong \,H^1_D(\Omega^1_X)$. By 
Lemma \ref{ph14}, there exists a natural morphism from 
$H^1_D(\Omega^1_X)$ to $H^1(\Omega^1_X)$. The image of $\{D\}$ in $H^1(\Omega^1_X)$
under this morphism is the cohomology class $[D]$ of $D$ (see \cite{fgag}).

\subsection{Inner multiplications with cohomology class of divisors}\label{inn}

Let $D \,\subseteq\, X$ be as in \S~\ref{sec:supp}. Denote by $\I_{D|X}$ the ideal sheaf of $D$ in $X$ and $i:D \hookrightarrow X$ the natural 
inclusion. Recall, 
\[\N_{D|X} \cong \Hc_X(\I_{D|X},\, i_*\mo_D) \mbox{ and } 
 \T_X \cong \Hc_X(\Omega^1_X,\, \mo_X).\]
Fix an 
open affine covering $\{U_i\}_{i \in I}$ as in \S\ref{sec:loc}
such that for each $i \in I$, the intersection $D \cap U_i$ is 
defined by a regular 
section, say $f_i \,\in\, \Gamma(U_i,\mo_X)$. 
On each $U_i$, a local section $\phi \in \Gamma(U_i,\T_X)$ 
(resp. $\phi \in \Gamma(U_i,\N_{D|X})$)
sends  $df_i \in \Gamma(U_i,\Omega^1_X)$ (resp. $f_i 
\in \Gamma(U_i,\I_{D|X})$) to $\phi(df_i) \in \Gamma(U_i, \mo_X)$
(resp. $\phi(f_i) \in \Gamma(U_i \cap D,\mo_D)$).
Using \eqref{brill03}, we have the short exact sequence
\begin{equation}\label{br5}
0 \,\longrightarrow\, \mo_X \,\longrightarrow\, j_*\mo_{X\backslash D} \,
\stackrel{\delta'}{\longrightarrow}\, \mc{H}^1_D(\mo_X) \,\longrightarrow\, 0\, .
\end{equation}
Define the contraction morphism
\[\lrcorner \{D\}\,:\,\mc{T}_X \,\longrightarrow\, \mc{H}^1_D(\mo_X) \, \, \, \, \, \, (\, \mr{respectively, }\, 
\, \, \lrcorner \{D\}'\,:\,\N_{D|X} \,\longrightarrow\, \mc{H}^1_D(\mo_X)),\]
which on each $U_i$ takes $\phi \in \Gamma(U_i,\T_X)$ (respectively, $\phi \in \Gamma(U_i,\N_{D|X})$)
to \[\delta'\left(\frac{\phi(df_i)}{f_i}\right) 
\in \Gamma(U_i,\mc{H}^1_D(\mo_X))\, \, \, \, \, \,
(\mbox{respectively, }\, \delta'\left(\frac{\widetilde{\phi(f_i)}}{f_i}\right)\, \in \Gamma(U_i,\mc{H}^1_D(\mo_X))),\]
where $\widetilde{\phi(f_i)}$ is a preimage of 
$\phi(f_i) \in \mo_D(U_i \cap D)$ under the
natural surjective morphism $$\mo_X(U_i) \,\longrightarrow\, \mo_D(U_i \, \cap \, D)\, .$$
Note that, any two pre-images of $\phi(f_i)$ differ by a section of 
$\I_{D|X}(U_i)$. Since $g/f_i$ is a regular section over $U_i$ for 
any section $g \in \I_{D|X}(U_i)$, \eqref{br5} implies
that $\delta'(g/f_i)=0$. 
As a result, $\lrcorner \{D\}'(\phi)$ does not 
depend on the choice of the pre-image. In other words,
$\lrcorner \{D\}'$ is well-defined. Similar argument shows that the 
morphisms $\lrcorner \{D\}$ and $\lrcorner \{D\}'$ does not depend 
on the choice of the local defining equations of $D$.

The morphisms $\lrcorner \{D\}$ and $\lrcorner \{D\}'$ induce on 
cohomology the morphisms 
\[\pi_D\,:\,H^1(\T_X)\,\longrightarrow \, H^1(\mc{H}^1_D(\mo_X)) \,\cong\, H^2_D(\mo_X) \mbox{ and }\]
\[\pi'_D\,:\,H^1(\N_{D|X}) \,\longrightarrow\, H^1(\mc{H}^1_D(\mo_X)) \,\cong\, H^2_D(\mo_X),\]
where the isomorphisms follow from Proposition \ref{ph15}.
As pointed out by the referee, the morphism $\pi'_D$ may
also be obtained by pairing $H^1(\N_{D|X})$ with 
$H^{n-2}(K_X|_D)$ and then applying formal duality. More 
precisely, by the adjunction formula (see \cite[Theorem III.$7.11$]{R1}), 
we have $\N_{D|X}^{\vee} 
\otimes K_D \cong K_X|_D$ (here $K_D$ denotes the dualizing sheaf of $D$).
By Serre duality, this implies 
\[H^1(\N_{D/X})^\vee\, \cong\, H^{n-2}(\N^\vee_{D/X}\otimes K_D)\, 
\cong\, H^{n-2}(K_X|_D). \]
By formal duality (see \cite[p. $48$, Proposition $5.2$]{hartr}),
we have an isomorphism of 
(not necessarily finite-dimensional) $\mb{C}$-vector spaces:
\[\varprojlim_m H^{n-2}(K_X \otimes \mo_X/\I_{D|X}^m) \,\cong \,
 H^2_D(\mo_X)^\vee,\]
where $\I_{D|X}$ is the ideal sheaf of $D$ in $X$.
By the definition of inverse limit, this gives us a natural morphism
\[H^2_D(\mo_X)^\vee \cong \varprojlim_m H^{n-2}(K_X \otimes \mo_X/\I_{D|X}^m)
\,\longrightarrow\, H^{n-2}(K_X|_D).\] As $H^{n-2}(K_X|_D)$ 
is isomorphic to $H^1(\N_{D|X})^\vee$, dualizing the resulting morphism 
from $H^2_D(\mo_X)^\vee$ to $H^1(\N_{D|X})^\vee$
will give us the morphism $\pi'_D$ above.

\subsection{Relations with the cup product}\label{sec:cup}
Let $X$ be a smooth, projective variety. Given two $\mo_X$-modules
$\mc{F}$ and $\mc{G}$,
there is a natural cup-product morphism 
\[H^p(\mc{F}) \otimes H^q(\mc{G}) \,\longrightarrow\, H^{p+q}(\mc{F} \otimes \mc{G})\, .\]
See \cite[\S~$5.3.2$]{v4} for a detailed description of this morphism.
In the case when $\mc{G}$ is the dual of $\mc{F}$, we can further 
compose this morphism by the natural pairing from $\mc{F} \otimes 
\mc{F}^{\vee}$ to $\mo_X$. More-precisely, we get a composed morphism
\[H^p(\mc{F}) \otimes H^q(\mc{F}^{\vee}) \,\longrightarrow\, H^{p+q}(\mc{F} \otimes 
 \mc{F}^{\vee}) \,\longrightarrow\, H^{p+q}(\mo_X)\, .\]
In the case, where $\mc{F}$ is the tangent sheaf $\T_X$, we have 
the morphism 
\[\bigcup\, :\, H^1(\T_X) \, \otimes\, H^1(\Omega^1_X) \,\longrightarrow\, H^2(\mo_X)\, .\]
We now observe that for a given effective divisor $D$ in $X$, the 
induced morphism $\bigcup (- \otimes [D])$, sometimes denoted by 
$- \bigcup [D]$, can be expressed in terms of the contraction morphism 
$\lrcorner \{D\}$ described in \S~\ref{inn} above. More precisely, 
using \cite[Proposition $6.2$]{b1}, we have 
for any $t \in H^1(\T_X)$ that $t \bigcup [D]$ coincides with 
the image of $t$ under the composition 
\[H^1(\T_X)\,\stackrel{\pi_D}{\longrightarrow}\,
H^2_D(\mo_X) \,\stackrel{\Phi_D}{\longrightarrow}\,
H^2(\mo_X)\, ,\]
where the morphism $\pi_D$ is induced by $\lrcorner \{D\}$
(see \S~\ref{inn}) and $\Phi_D$ follows from Lemma \ref{ph14}. 

\subsection{Hodge locus for divisors}\label{sec:hod}

Consider a family \[\pi\,:\,\mc{X} \,\longrightarrow\, B\] of smooth, projective 
varieties with a reference point $o \,\in \, B$. Let $X\,:=\,\pi^{-1}(o)$ be the
inverse image.
The differential of $\pi$ produces a short exact sequence of sheaves
\begin{equation}\label{fe1}
0 \,\longrightarrow\, \mc{T}_X \,\longrightarrow\, \mc{T}_{\mc{X}}\vert_X 
\,\longrightarrow\, \pi^* T_oB \,\longrightarrow\, 0\, ,
\end{equation}
where $T_oB$ is the tangent space to $B$
at $o$. The \emph{Kodaira--Spencer map} $$\mr{KS}\,:\,T_oB \,\cong\, H^0(\mc{X},\, \pi^*(T_oB))
\,\longrightarrow\, H^1(\mc{X},\, \mc{T}_X)$$ is the 
coboundary morphism in the
long exact sequence of cohomologies associated to the
short exact sequence in \eqref{fe1}.

Let $D$ be an effective divisor in $X$. 
The \emph{Hodge locus associated to the cohomology class} $[D]$, denoted $\NL([D])$,
is the subspace of $B$ consisting of points $b \,\in\, 
B$ such that $[D]$ deforms to a Hodge class on $\mc{X}_b\,=\,\pi^{-1}(b)$. By the 
Lefschetz $(1,1)$-theorem, every cohomology class of a divisor 
arises as the first Chern class of a line bundle. 
As a result deformations of the cohomology class $[D]$ is equivalent to the deformation of the line bundle $\mo_X(D)$. Hence, 
the subspace $\NL([D])$ can
also be seen as a subspace of $B$ over which
the invertible sheaf corresponding to $D$ deforms as an invertible sheaf. See
\cite[\S~$5.3.1$]{v5} for a scheme-theoretic description of $\NL([D])$.

Denote by $T_o\NL([D])$ the tangent space at $o$ of $\NL([D])$.
By \cite[Theorem $10.21$]{v4}, the tangent space $T_o\NL([D])$ to the Hodge locus 
$\NL([D])$ at the point $o$ consists of all 
$t\, \in \, T_oB$ such that $[D] \, \bigcup \, \mr{KS}(t)\, =\, 0$.
Given any $t \,\in\, T_oB$, we say that $[D]$ \emph{deforms as a 
Hodge class of} $X_t$ if $$[D] \bigcup \mr{KS}(t)\,=\,0\,,$$ where 
$X_t$ denotes the first order infinitesimal deformation of $X$ in the direction
of $t$.

\section{Torsion geometric obstruction}

Let $X$ be a smooth, projective variety and $D\,\subseteq\, X$
be an effective divisor.
In this section, we associate to a first order
infinitesimal deformation $X_t$ of $X$
an element of $H^2_D(\mo_X)$, the vanishing of which is equivalent to the deformation 
of a multiple of $D$ as an effective Cartier divisor of $X_t$ (Theorem 
\ref{brill04}).

\subsection{Cokernel of the inner multiplication $\lrcorner \{D\}'$}
Note that, there is a natural morphism from $\mo_X$ to $\mo_X(D)$
defined by multiplication by a global section, say
$s_D \,\in\, H^0(\mo_X(D))$.  Then, for any $n>0$, the natural morphism 
from $\mo_X$ to $\mo_X(nD)$ is defined by multiplication by $s_D^n$.
For simplicity of notation, we denote by 
\[s_D^n\,:\, \mc{H}^1_D(\mo_X) \,\longrightarrow\, \mc{H}^1_D(\mo_X(nD))\]
the morphism induced by multiplication by $s_D^n$. We then 
observe that the morphism $s_D^n$ is surjective with kernel isomorphic 
to $\N_{nD|X}$:

\begin{lem}\label{lef02}
Let $X$ be a smooth projective variety, and let $D\,\subseteq\, X$ be an effective divisor. 
For every $n \ge 1$, there is a short exact sequence of the form 
\begin{equation}\label{brill02}
0 \,\longrightarrow\, \N_{nD|X}\,\xrightarrow{\lrcorner \{nD\}'}\,
\mc{H}^1_D(\mo_X) \,\xrightarrow{s_D^n}\,
\mc{H}^1_D(\mo_X(nD)) \,\longrightarrow\, 0,
\end{equation}
where $\lrcorner \{nD\}'$ is the inner multiplication described in \S~\ref{inn}.
\end{lem}

\begin{proof}
Let $j\,:\,X \backslash D \,\longrightarrow\, X$ be the natural inclusion. Consider the
commutative diagram of short exact sequences
\begin{equation}\label{lef01}
\xymatrix{
0 \ar[r] & \mo_X \ar[r]^{\hspace{15mm}} \ar[d]^{s_D^n} & j_*\mo_{X\backslash D} \ar[d]^{\mr{id}} \ar[r] & \mc{H}^1_D(\mo_X) \ar[d]^{s_D^n} \ar[r] & 0\\
0 \ar[r] & \mo_X(nD) \ar[r] & j_*\mo_{X\backslash D} \ar[r] 
& \mc{H}^1_{D}(\mo_X(nD)) \ar[r] & 0 }
\end{equation}
where the first horizontal exact sequence is the one in
\eqref{br5} and the second one is obtained after 
replacing $\mo_X$ in \eqref{br5} by $\mo_X(nD)$ (use the identification $\mo_{X \backslash D}
\,\cong\, \mo_X(nD)|_{X \backslash D}$). Let $\{U_i\}_{i \in I}$ be an affine open covering 
of $X$ such that for each $i \,\in\, I$, the
intersection $D\cap U_i$ is defined by exactly one equation, say $f_i\, \in
\,\Gamma(U_i,\,\mo_X)$; the equation is unique up to multiplication by units.
Consider now the short exact sequence
 \[0 \,\longrightarrow\, \mo_X \,\xrightarrow{s_D^n}\, \mo_X(nD) \,\longrightarrow\, \N_{nD|X}
\,\longrightarrow\, 0\, ,\]
where the morphism from $\mo_X(nD)$ to $\N_{nD|X}$ is defined over each open set $U_i$
by sending any $g \,\in \,\Gamma(U_i,\,\mo_X(nD))$ to the morphism $\phi \,\in\,
\mr{Hom}_{_{U_i}}(\mo_X(-nD)|_{U_i}, \,\mo_{nD \cap U_i})$ satisfying the condition
 \[\phi(f_i^n)\,=\,gf_i^n \ \ \mod\ \mo_X(-nD)|_{U_i}\, .\]
 Using the Snake lemma and diagram chase applied to \eqref{lef01}, we then get the 
 short exact sequence
 \[0 \,\longrightarrow\, \N_{nD|X}\,\xrightarrow{\lrcorner \{nD\}'}\,
\mc{H}^1_D(\mo_X) \,\xrightarrow{s_D^n}\,
\mc{H}^1_D(\mo_X(nD)) \,\longrightarrow\, 0.\]
This proves the lemma.
\end{proof}

\subsection{Geometric obstruction}\label{sec:top}
Let $\pi\,:\,\mc{X}\,\longrightarrow\, B$ be a flat family of smooth projective 
varieties with a fixed base point $o \,\in\, B$. Let $\mc{X}_b\,:=\,\pi^{-1}(b)$
be the inverse image of any $b \,\in \,B$, and set $X\,:=\,\pi^{-1}(o)$.

\begin{note}\label{note:lef02}
Let $D \,\subseteq\, X$ be an effective divisor. 
There is a natural morphism 
\[\rho_{_D}\,:\, \T_X \,\longrightarrow\, \N_{D|X}\]
obtained as a composition of the restriction morphism from $\T_X$
to $\T_X|_D$ with the morphism from $\T_X|_D$ to $\N_{D|X}$ arising
as the dual of the restriction to $D$ of the morphism 
\[\mo_X(-D) \,\stackrel{s_D}{\longrightarrow}\,
\mo_X\,\stackrel{d}{\longrightarrow}\, \Omega^1_X,\]
where $s_D$ is the global section of $\mo_X(D)$ corresponding to the 
divisor $D$. Let
\[\mr{Ob}_D\, :\, H^1(\T_X)\, \longrightarrow\, H^1(\N_{D|X})\]
be the homomorphism of cohomology groups induced by $\rho_{_D}$. 
\end{note}

Recall, the morphism $\pi_D$ from $H^1(\T_X)$ to $H^2_D(\mo_X)$
obtained by applying the cohomology functor to the contraction 
morphism $\lrcorner \{D\}$ from $\T_X$ to $\mc{H}^1_D(\mo_X)$
described in \S~\ref{inn}. We now observe that the image of $\pi_D$
contains the obstruction to the deformation of a multiple of $D$.

\begin{thm}\label{brill04}
Let $t \,\in \,T_oB$ and $D\, \subseteq\, X$ an effective divisor. 
Then, $\pi_D \circ \mr{KS}(t)\,=\,0$ if and only if there exists an integer $n\,>\,0$ such that $nD$ 
deforms as an effective Cartier divisor in the first order 
infinitesimal deformation $X_t$ of $X$ corresponding to $t$. 
\end{thm}

\begin{proof}
Take any $t \,\in\, T_oB$. Denote by $X_t$ the first order infinitesimal 
deformation of $X$ corresponding to $t$. Recall, for any
positive integer $n\,>\,0$, the effective divisor $nD$ deforms as an 
effective Cartier divisor in $X_t$ if and only if $\mr{Ob}_{nD}
\circ \mr{KS}(t)$ vanishes (see \cite[Theorem $6.2$]{R3}). Therefore,
to prove the theorem, we need to compare the obstruction
morphisms $\mr{Ob}_D$ and $\mr{Ob}_{nD}$, where $n$ is some positive
integer. Let $s_D$ be the global section of $\mo_X(D)$ corresponding to 
the divisor $D$.
Now, for any regular local section $f$ of $\mo_X$, we have
\[d(fs_D^n)\,=\,ns_D^{n-1}fd(s_D) + s_D^ndf.\]
Restricting to $nD$, the last term of the equality vanishes.
We then get the commutative diagram:
\begin{equation}\label{diag01}
\xymatrixcolsep{5pc}
\xymatrix{
\T_X\, \ar[r]^{\rho_{_D}} \ar[d]^{\times n}_{\wr}\, &\, \N_{D|X}\, \ar[d]^{s_D^{n-1}}\\
\T_X \, \ar[r]^{\rho_{_{nD}}}\, & \, \N_{nD|X}}
\end{equation}
where the morphism from $\N_{D|X}$ to 
$\N_{nD|X}$ is induced by multiplication by $s_D^{n-1}$. This relates
$\mr{Ob}_D$ to $\mr{Ob}_{nD}$. Now, by
Lemma \ref{lef02} we have the following diagram of 
short exact sequences:
{\small \begin{equation}\label{diag02}
 \xymatrixcolsep{5pc}
\xymatrix{
0 \ar[r] & \N_{D|X} \, \ar[r]^{\lrcorner \{D\}'} \ar[d]^{s_D^{n-1}}\, &\, \mc{H}^1_D(\mo_X)\, \ar[d]^{\mr{id}}_{\wr} \ar[r]^{s_D} \, & \, \mc{H}^1_D(\mo_X(D))\, \ar[r] \ar[d]^{s_D^{n-1}} & 0 \\
0 \ar[r] & \N_{nD|X} \, \ar[r]^{\lrcorner \{nD\}'}\, & \, \mc{H}^1_D(\mo_X)\, \ar[r]^{s_D^n}\, &\, \mc{H}^1_D(\mo_X(nD)) \ar[r] & 0}
\end{equation}}
Note that, the contraction morphism $\lrcorner \{D\}$
from $\T_X$ to $\mc{H}^1_D(\mo_X)$ factors as 
\[\T_X \xrightarrow{\rho_{_D}} \N_{D|X} \xrightarrow{
 \lrcorner \{D\}'} \mc{H}^1_D(\mo_X).\]
 This implies $\pi'_D \circ \mr{Ob}_D=\pi_D$.
The long exact sequence associated to \eqref{diag02}
then implies that $\pi_D \circ \mr{KS}(t)\,=\,0$
if and only if there exists $f \in H^0(\mc{H}^1_D(\mo_X(D)))$
such that the image of $f$ under the coboundary map 
\[H^0(\mc{H}^1_D(\mo_X(D))) \,\longrightarrow\, H^1(\N_{D|X})\]
equals $\mr{Ob}_D \circ \mr{KS}(t)$.
Observe that given any global section $f$ of 
$\mc{H}^1_D(\mo_X(D))$ there exists an integer $N\,>\,0$ such that 
for all $n \ge N$, $fs_D^{n-1}$ extends to a 
regular section of $\mo_X(nD)$, 
hence defines a zero section in 
$H^0(\mc{H}^1_D(\mo_X(nD)))$ by the exact sequence given in 
Lemma \ref{ph14}. 
The theorem then follows by a simple diagram chase of the diagram 
of long exact sequences arising from \eqref{diag02}.
This proves the theorem.
\end{proof}

Theorem \ref{brill04} motivates the following definition.

\begin{defi}
For any $t \,\in \,T_oB$, denote by $X_t$ the first order infinitesimal deformation of $X$ 
corresponding to $t$.
Let $D \,\subseteq\, X$ be an effective divisor. We say that $D$ \emph{is saturated}, along $B$,
if for every $t \,\in\, T_oB$, the divisor $D$ lifts as
an effective Cartier divisor in $X_t$ if and only if $nD$ lifts as an effective 
Cartier divisor in $X_t$ for every $n \,\ge \,1$.
\end{defi}

\begin{cor}\label{cor:sat}
Let $D \,\subseteq\, X$ be an effective divisor. Then $D$ is saturated if and only if 
\[\Ima(\mr{Ob}_D \circ \mr{KS}) \cap \ker(\pi'_D\,:\, H^1(\N_{D|X}) \,\longrightarrow\, 
H^1(\mc{H}^1_D(\mo_X)))=0. \]
\end{cor}

\begin{proof}
Using Theorem \ref{brill04}, note that the intersection of 
$\Ima(\mr{Ob}_D \circ \mr{KS})$ with $\ker(\pi'_D)$
 consists of all elements of the form 
$\mr{Ob}_D \circ \mr{KS}(t)$ such that there exists $n\,>\,0$ for which 
$nD$ deforms to an effective Cartier divisor in the first order 
infinitesimal deformation $X_t$
of $X$ corresponding to $t$, but $D$ does not deform. Hence, the divisor $D$ is saturated if and only if 
\[\Ima(\mr{Ob}_D \circ \mr{KS}) \cap \ker(\pi'_D)=0.\]
This proves the corollary.
\end{proof}

\section{Local topological obstruction}\label{sec:ltop}

Throughout this section, $\pi\,:\,\mc{X}\,\longrightarrow\, B$
stands for a flat family of smooth projective 
varieties with a fixed base point $o \,\in\, B$. Let $\mc{X}_b\,:=\,\pi^{-1}(b)$ 
be the inverse image of any $b \,\in\, B$, and set $X\,:=\,\pi^{-1}(o)$. 
Let $D\,\subseteq\, X$ be an effective Cartier divisor. As discussed in the 
introduction, the topological obstruction
to deforming the fundamental class of $D$ 
as a Hodge class, lie in $H^2(\mo_X)$. The goal of this section is to 
replace $H^2(\mo_X)$ by $H^2_D(\mo_X)$. For this we consider the 
natural morphism $\Phi_D$ from $H^2_D(\mo_X)$ to $H^2(\mo_X)$. 
We observed in Theorem 
\ref{brill04} that for any $t \,\in \,T_oB$, 
the element $\pi_D \circ \mr{KS}(t)$ 
measures the obstruction to 
deforming a multiple of $D$ as an effective Cartier divisor. It is of course 
possible for $\pi_D \circ \mr{KS}(t)$ to be non-zero, 
but $\Phi_D \circ \pi_D \circ \mr{KS}(t)$ to vanish
(see Remark \ref{rem:ref}). However, we observe that all such 
$\pi_D \circ \mr{KS}(t)$ 
come from the $1$-cocycle of the structure sheaf of the complement of 
$X$ by a codimension $2$ subvariety $T$ (see Theorem \ref{th:br03}).
We study this further in Theorem \ref{brnew01} and Corollary \ref{brill09}.

\subsection{Topological obstruction}

Let $D \,\subseteq\, X$ be an effective divisor. Recall the morphism $\pi_D$
from $H^1(\T_X)$ to $H^2_D(\mo_X)$ defined in \S~\ref{inn}. Let 
\[\mr{Obs}_{D,\mr{loc}}^{\pi}\,:=\,
\mr{Im}(\pi_D \circ \mr{KS})\, \subseteq H^2_D(\mo_X) \ \text{ and }\ \
\mr{Obs}_D^{\pi}\,:=\,
\mr{Im}(\Phi_D \circ \pi_D \circ \mr{KS})\, \subseteq H^2(\mo_X)\]
be the topological obstruction spaces associated to $D$, where $\Phi_D$
is the natural morphism from $H^2_D(\mo_X)$ to $H^2(\mo_X)$.
It follows immediately 
that there exists a natural surjective morphism from 
$\mr{Obs}^{\pi}_{D,\mr{loc}}$ to $\mr{Obs}^{\pi}_D$, induced by the 
morphism $\Phi_D$. It is possible that this morphism is not injective.
In fact, we now give an example when this map is the zero map.

\begin{rem}\label{rem:ref}
The example arises from the blow-up of two points, 
say $p, q$ in $\mb{P}^2$, as $p$ approaches $q$. More precisely,
set $Y\,:=\,\mb{P}^2 \times \mb{A}^1$ and take $T_1 \,\subseteq\, Y$
(respectively, $T_2 \,\subseteq\, Y$) closed subvarieties consisting of points of the 
form $([1:0:t],\,t)$ (respectively, $([1:t:0],\,t)$) as $t$ varies over $\mb{A}^1$.
Let $Y_1$ be the blow-up of $Y$ along the closed subscheme $T_1$, and let $E_1$
be the associated exceptional divisor. 
Let $T_2'$ be the strict transform of $T_2$ in $Y_1$. Denote by 
\[p:Y_2 \to Y_1\]
the blow-up of $Y_1$ along $T_2'$, and let $E_2\, \subseteq\, Y_2$ be the exceptional divisor. 
Let $Y_{2,0}$ be the fiber
over $0 \,\in\, \mb{A}^1$ of the natural
morphism \[\pi: Y_2\, \longrightarrow\, \mb{A}^1.\]
Note that the generic fiber of $\pi$ is the blow-up of $\mb{P}^2$
at two distinct points and in the limit 
the central fiber has an $A_2$-configuration
of $\mb{P}^1$s.
Let,
$E_{1,0}$ and $E_{2,0}$ respectively are the fibers
over $0 \,\in\, \mb{A}^1$ of the natural
morphisms $E_1\, \longrightarrow\, \mb{A}^1$ and $E_2\, \longrightarrow\, \mb{A}^1$.
Note that, $E_{1,0}$ is isomorphic to $\mb{P}^1$.
Denote by $D$ the strict transform of $E_{1,0}$ in $Y_2$.
Clearly, $D$ is isomorphic to $\mb{P}^1$. 
Using 
\cite[Propositions V.$3.2$ and V.$3.6$]{R1}, we conclude that
\[D^2+1\,=\,(D+E_{2,0}).D\,=\,p^*(E_{1,0}).D\,= \,(E_{1,0})^2\,=\,-1\, .\]
Hence, $D^2\,=\,-2$.  This implies $\N_{D|Y_{2,0}} \cong \mo_D(-2)$.
As $K_D \cong \mo_D(-2)$, we have by Serre duality 
\[H^1(\N_{D|Y_{2,0}}) \cong H^0(\mo_D)^\vee \cong 
\mb{C}.\]
%
By \cite[Proposition V.$3.4$]{R1}, we have $H^2(\mo_{Y_{2,0}})\,=\,0$.
This means that $\mr{Obs}^\pi_D\,=\,0$.  

We will now observe that $\mr{Obs}^\pi_{D,\mr{loc}}
\,=\,\mb{C}$. For this purpose, we will show that 
$D$ or any positive multiple of $D$ does not deform even 
upto first order, as an effective 
Cartier divisor, as $Y_{2,0}$ deforms along $\mb{A}^1$ in the family 
$\pi$. Using Theorem \ref{brill04}, this will imply $\mr{Obs}^\pi_{D,\mr{loc}} \subseteq H^2_D(\mo_X)$ is non-trivial. 
As $H^1(\N_{D|Y_{2,0}})=\mb{C}$, it will follow from definition that 
$\mr{Obs}^\pi_{D,\mr{loc}}\,=\,\mb{C}$.
We  now prove the obstruction 
to the first-order infinitesimal deformation of $nD$ for any $n>0$.
As $Y_{2,0}$ is a fiber over $0 \,\in\, \mb{A}^1$,
we have $\N_{Y_{2,0}|Y} \,\cong\, \mo_{Y_{2,0}}$.
Then, the normal short exact sequence is of the form:
\begin{equation}\label{eq01}
 0 \,\longrightarrow\, \N_{D|Y_{2,0}}\,\longrightarrow\, \N_{D|Y_2}
\,\longrightarrow\, \mo_D \,\longrightarrow\, 0\,
\end{equation}
Since $\N_{D|Y_{2,0}} \cong \mo_D(-2)$,
only one of the following is possible:
\begin{itemize}
 \item Either $\N_{D|Y_2}$ is isomorphic to $\mo_D(-2) \oplus 
\mo_D$, in which case \eqref{eq01} splits,
\item or $\N_{D|Y_2}$ is isomorphic to $\mo_D(-1) \oplus \mo_D(-1)$.
\end{itemize}
Note that, $D$  can be contracted to a
$3$-fold with ordinary double point given by the blow-up of 
$Y$ along $T_1 \cup T_2$. This implies that 
$N_{D|Y_2}$ must be isomorphic to 
$\mo_D(-1) \oplus \mo_D(-1)$ (see \cite[p. $158, 159$]{reid}).
The obstruction to the first order deformation of 
$D$ (as an effective Cartier divisor) as $Y_{2,0}$ deforms 
along $\mb{A}^1$, is the image of $1 \in H^0(\mo_D)$ under the 
coboundary map:
\[H^0(\mo_D) \,\longrightarrow\, H^1(\N_{D|Y_{2,0}})\, \]
arising from \eqref{eq01}.
Since $H^0(\mo_D(-1) \oplus \mo_D(-1))\,=\,0$, the coboundary map 
is injective. Hence, the obstruction does not vanish.
In other words, $D$ does not deform to first order as an 
effective Cartier divisor, as $Y_{2,0}$ deforms along $\mb{A}^1$.
We now prove the same for higher multiples of $D$.
Since $D^2\,=\,-2$, we have $H^0(\mo_D \otimes \mo_{Y_{2,0}}(mD))\,=\,0$ for all $m\,>\,0$.
Then for any $m\,>\,0$, the long exact sequence 
associated to short exact sequence:
\[0\,\longrightarrow\, \mo_{Y_{2,0}}((m-1)D)\,\longrightarrow\, \mo_{Y_{2,0}}(mD)
\,\longrightarrow\,\mo_D \otimes \mo_{Y_{2,0}}(mD)\,\longrightarrow\, 0\]
implies that $$H^0(\mo_{Y_{2,0}}((m-1)D))\,=\,H^0(\mo_{Y_{2,0}}(mD))$$ and the morphism from 
$H^1(\mo_{Y_{2,0}}((m-1)D))$ to $H^1(\mo_{Y_{2,0}}(mD))$ is injective. Recursion on 
$m$ implies that for any $n\,>\,0$:
\begin{itemize}
\item $\mb{C}\,=\,H^0(\mo_{Y_{2,0}}(D))\,=\, H^0(\mo_{Y_{2,0}}(nD))$, and

\item the natural morphism from $H^1(\mo_{Y_{2,0}}(D))$ to 
$H^1(\mo_{Y_{2,0}}(nD))$ is injective.
\end{itemize}
Denote by $\mc{L}_t$ the first order 
infinitesimal deformation of $\mo_{Y_{2,0}}(D)$ (along 
the family $\pi$). By standard
deformation theory, we have 
the following commutative diagram for any $n>0$:
\[\xymatrixcolsep{5pc}
\xymatrix{
 H^0(\mc{L}_t) \, \ar[r] \ar[d]\, &\, H^0(\mo_{Y_{2,0}}(D))\, \ar[d]_{\wr} \ar[r] \, & \, H^1(\mo_{Y_{2,0}}(D))\, \ar@{^{(}->}[d] \\
 H^0(\mc{L}_t^{\otimes n}) \, \ar[r] \, &\, H^0(\mo_{Y_{2,0}}(nD))\, \ar[r] \, & \, H^1(\mo_{Y_{2,0}}(nD))}\]
Note that there exists a first order infinitesimal deformation of $nD$
(respectively, $D$)
if and only if the corresponding section in $H^0(\mo_{Y_{2,0}}(nD))$
(respectively, $H^0(\mo_{Y_{2,0}}(D))$)
comes from a section of $H^0(\mc{L}_t^{\otimes n})$ (respectively, 
$H^0(\mc{L}_t)$). Using the diagram, we then conclude that 
since $D$ does not deform as an effective Cartier divisor up to first
order, neither does $nD$ for any $n>0$.
As explained above, we then conclude that 
$\mr{Obs}^{\pi}_{D,\mr{loc}}\,=\,\mb{C}$.
\end{rem}

We will now observe that the difference between $\mr{Obs}^\pi_D$ 
and $\mr{Obs}^{\pi}_{D,\mr{loc}}$ arises from $1$-cocycles of the structure
sheaf of $X$, after removing a codimension $2$ subvariety. More precisely,
let $T \,\subseteq\, D$ be a proper, closed subvariety (by proper we 
mean that $\dim_x T\,<\,\dim_x D$ for all $x \in T$) of $D$. Let
\[
\Phi_{(T \subseteq D)}\,:\, H^1(\mo_{X\backslash T}) \,\longrightarrow\, H^2_D(\mo_X)
\]
be the composition of homomorphisms
\[
H^1(\mo_{X\backslash T})
\,\longrightarrow\, H^1(\mo_{X\backslash D})\,\longrightarrow\, H^2_D(\mo_X)\, ,\]
 where the first morphism is natural and the second morphism 
 arises from the short exact sequence \eqref{br5} (note that,
 $H^1(\mo_{X \backslash D}) \,\cong\, H^1(j_*(\mo_{X \backslash D}))$
 as $j$ is an affine morphism). From the exactness of the 
 sequence \[H^1(\mo_{X\backslash D})\,\longrightarrow\, H^2_D(\mo_X)\,\longrightarrow\, H^2(\mo_X)\]
it follows that the cokernel of the morphism from 
 $H^1(\mo_{X\backslash D})$ to $H^2_D(\mo_X)$ is naturally contained
 in $H^2(\mo_X)$. This induces a natural map:
 \begin{equation}\label{eq:eta}
\eta_{_{T}}\,:\, \mr{coker}(\Phi_{(T \subseteq D)})\,\longrightarrow\, H^2(\mo_X)\, .
\end{equation}

\begin{thm}\label{th:br03}
There exists a non-trivial, proper, closed 
subvariety $T \,\subseteq\, D$, for which the image of 
$\mr{Obs}^{\pi}_{D,\mr{loc}}$ in $\mr{coker}(\Phi_{(T \subseteq D)})$
maps isomorphically to $\mr{Obs}^{\pi}_D$, under the 
natural morphism $\eta_{_{T}}$ mentioned above.
 \end{thm}

\begin{proof}
Let $E$ be an effective divisor on $X$. Set $T\,:=\,D.E$.
We then have the following commutative diagram of exact sequences
where both the horizontal rows and the vertical columns are exact:
\begin{equation}\label{brill16}
\xymatrix{
H^1(\mo_{X\backslash T}) \ar[r]^{\delta_3 \hspace{15mm}} \ar[d]^{\Phi'_T} & H^1(\mo_{X\backslash D}) \, \oplus\, H^1(\mo_{X\backslash E}) \ar[d]^{\Phi'_D \, \oplus\, \Phi'_E} \ar[r]^{\hspace{10mm} \delta_4} & H^1(\mo_{X\backslash {D \, \cup \, E}}) \ar[d]^{\Phi'_{D \, \cup \, E}} \\
H^2_T(\mo_X) \ar[r]^{\delta_1 \hspace{15mm}} \ar[d]^{\Phi_T} &\, H^2_D(\mo_X)\, \oplus\, H^2_E(\mo_X) \ar[d]^{\Phi_D\, \oplus\, \Phi_E} \ar[r]^{\hspace{10mm} \delta_2} &\, H^2_{D \, \cup \, E}(\mo_X) \ar[d]^{\Phi_{D\, \cup\, E}}\\
H^2(\mo_X) \ar@{^{(}->}[r]^{i\hspace{10mm}} &\, H^2(\mo_X)\, \oplus \, H^2(\mo_X) \ar[r]^{\hspace{6mm}r} &\, H^2(\mo_X) \ar[r] &\, 0 }
 \end{equation}
 In \eqref{brill16}, the top two rows are the Mayer--Vietoris sequences (see \cite[Ex. III.$2.4$]{R1}),
$i(a)\, =\, a \, \oplus\, a$, $r(a \, \oplus\, b)\, =\, a-b$ and $\Phi_{(-)}, \, \Phi'_{(-)}$ are
natural morphisms arising from Lemma \ref{ph14}. We want to 
show that in the case when $E$ is sufficiently ample, if there 
exists a non-zero element $\xi$ in the image of $\pi_D \circ \mr{KS}$
which maps to zero under $\Phi_D$, then $\xi$ lies in the image of 
$\Phi_{(T \subseteq D)}\,=\,\Phi'_D \circ \delta_3$.

As $E$ is sufficiently ample (in particular, $\mo_X(D)(E)$ and $\mo_X(E)$
are very ample), we have $X \backslash E$ and $X \backslash (D \cup E)$
are affine. This implies 
\[H^1(\mo_{X \backslash E})\,=\,0\,=
\,H^1(\mo_{X \backslash (D \cup E)}). \]
Hence, $\Phi_E$ and $\Phi_{D \cup E}$ are injective morphisms.
This implies that for any $\xi \in \ker(\Phi_D)$, we have 
 $\delta_2(\xi \oplus 0)\,=\,0$. Hence, by the exactness of the 
middle horizontal row of \eqref{brill16}, there exists 
$\xi_1' \,\in\, H^2_T(\mo_X)$ such that $\delta_1(\xi_1')\,=\,\xi \oplus 0$.
By the injectivity of $i$, note that $\Phi_T(\xi_1')\,=\,0$.
Then, by the exactness of the first vertical column of 
\eqref{brill16}, there exists $\xi' \,\in\, H^1(\mo_{X\backslash T})$
such that $\Phi'_T(\xi')\,=\, \xi_1'$. By the commutativity of the 
upper left hand square in \eqref{brill16}, we have that 
\[\xi \oplus 0\,=\,\delta_1 \circ \Phi'_T(\xi')\,=\,(\Phi'_D \oplus 
 \Phi'_E) \circ \delta_3(\xi').\]
In particular, $\Phi'_D \circ \delta_3(\xi')\,=\,\xi$.
Note that, for $T\, =\, D.E$, we have (by definition)
$\Phi'_D \circ \delta_3\, =\, \Phi_{T \subseteq D}$.
This means, for any $\xi \in \ker(\Phi_D)$, there exists 
$\xi' \in H^1(\mo_{X\backslash T})$ such that $\Phi_{T \subseteq D}(\xi')
\,=\, \xi$. This implies that the image of 
$\mr{Obs}^{\pi}_{D,\mr{loc}}$ in $\mr{coker}(\Phi_{(T \subseteq D)})$
maps isomorphically to $\mr{Obs}^{\pi}_D$, under the 
natural morphism $\eta_{_{T}}$ mentioned above.
This proves the theorem.
\end{proof}

The above theorem motivates the following definition:
 
\begin{defi}
Given a proper, closed subvariety $T\,\subseteq\, D$, we say that $D$ is 
\emph{$T$-semi-regular (along $B$)} if the conclusion of 
Theorem \ref{th:br03} holds for such $T$ i.e., 
the image of 
$\mr{Obs}^{\pi}_{D,\mr{loc}}$ in $\mr{coker}(\Phi_{(T \subseteq D)})$
maps isomorphically to $\mr{Obs}^{\pi}_D$, under the 
natural morphism $\eta_{_{T}}$ as in \eqref{eq:eta}.
\end{defi}

The question we now want to study is:

{\bf{Question}}: Given an effective divisor $D$ in 
$X$, what are all the possible choices of proper 
closed subvarieties $T\,\subseteq\,D$ such that $D$ is $T$-semi-regular?
 
\subsection{Comparison with the geometric obstruction}

Let $\pi$ be the family of smooth, projective varieties 
described above and $D \,\subseteq\, X$ be an effective divisor.
Let
\[\mbf{O}_D^{\pi}\,:=\, \Ima(\mr{Ob}_D) \circ\mr{KS}
\subseteq H^1(\N_{D|X})\] be the 
obstruction space for deforming $D$ as an effective Cartier 
divisor. Recall that the morphism $\pi'_D$ from
$H^1(\N_{D|X})$ to $H^2_D(\mo_X)$ induced by the contraction morphism
$\lrcorner \{D\}'$ as defined in \S~ \ref{inn}.
Let
\[\mr{Tors}_D\,:=\,\ker(\mbf{O}_D^\pi \xrightarrow{\pi'_D} H^2_D(\mo_X))\]
be the kernel.
By Theorem \ref{brill04}, $\mr{Tors}_D$ encodes information on 
infinitesimal 
deformations of $X$ such that $D$ does not deform as an effective Cartier
divisor, but some multiple of $D$ does deform. By the first isomorphism
theorem, the 
morphism $\pi'_D$ induces an isomorphism from $\mbf{O}_D^{\pi}/\mr{Tors}_D$
to $\mr{Obs}^{\pi}_{D,\mr{loc}}$. Then, the following composed 
morphism is surjective:
\[ \mbf{O}_D^{\pi}/\mr{Tors}_D \,\xrightarrow[\sim]{\pi'_D}\,
\mr{Obs}^{\pi}_{D,\mr{loc}} \,\xrightarrow{\Phi_D} \,\mr{Obs}^{\pi}_D\, .\]
We say that $D$ is \emph{effective Lefschetz} if this composed morphism
is an 
isomorphism (this means that the fundamental class of $D$ deforms as a Hodge class 
if and only if a multiple of $D$ deforms as an effective Cartier divisor).
We now give a partial answer to the question posed above.

\begin{thm}\label{brnew01}
Let $D \,\subseteq\, X$ be an effective divisor. The following two hold:
\begin{enumerate}
\item $D$ is $T$-semi-regular for every
proper subvariety $T \,\subseteq\, D$ if and only if $D$ is effective Lefschetz,

\item let $\mc{L}$ be an invertible sheaf on $\mc{X}$ and 
$E \,\in \,|\mc{L}_o|$ an effective divisor intersecting $D$ transversally,
such that the union $D \cup E$ is effective Lefschetz.
Then, $D$ is $T$-semi-regular, where $T$ is the intersection product $D.E$.
\end{enumerate}
\end{thm}

\begin{proof}
 $(1)$ Suppose that $D$ is $T$-semi-regular for every proper 
 closed subvariety $T\,\subseteq \,D$, in particular, in the case 
 $T\,=\,\emptyset$. If $T\,=\,\emptyset$, then $H^1(\mo_{X\backslash T})
\, =\,H^1(\mo_X)$ and the morphism $\Phi_{(T \subseteq D)}$ coincides 
 with the composition:
 \[H^1(\mo_X) \,\longrightarrow\, H^1(\mo_{X\backslash D})\,\longrightarrow\, H^2_D(\mo_X)\]
 which is the zero map, as this is an exact sequence (see 
 Lemma \ref{ph14}). Hence, for $T=\emptyset$, we have 
 $\mr{coker}(\Phi_{(T \subseteq D)}) \cong H^2_D(\mo_X)$ and 
the morphism $\eta_{_{T}}$ from $\mr{coker}(\Phi_{(T \subseteq D)})$ to 
 $H^2(\mo_X)$, as in \eqref{eq:eta},
coincides with $\Phi_D\,:\, H^2_D(\mo_X) \,\longrightarrow\, H^2(\mo_X)$.
Then, by the definition of 
$T$-semi-regularity for $T\,=\,\emptyset$, we have that $\Phi_D$ maps 
$\mr{Obs}^{\pi}_{D,\mr{loc}}$ isomorphically to $\mr{Obs}^{\pi}_D$, i.e.,
$D$ is effective Lefschetz.
 
Conversely, suppose that $D$ is effective Lefschetz. Note that, for any 
proper closed subset $T\,\subseteq\, D$, the morphism $\Phi_D$ factors as:
\[H^2_D(\mo_X) \,\longrightarrow\, \mr{coker}(\Phi_{(T \subseteq D)})\, \xrightarrow{\eta_{_{T}}} 
\, H^2(\mo_X)\, . \]
 By assumption, $\Phi_D$ maps $\mr{Obs}^{\pi}_{D,\mr{loc}}$ isomorphically 
 to $\mr{Obs}^{\pi}_D$. Hence, 
 the image of 
$\mr{Obs}^{\pi}_{D,\mr{loc}}$ in $\mr{coker}(\Phi_{(T \subseteq D)})$
maps isomorphically to $\mr{Obs}^{\pi}_D$, under the 
morphism $\eta_{_{T}}$. Therefore, $D$ is $T$-semi-regular for any 
proper closed subset $T \,\subseteq\, D$.
This proves $(1)$.

$(2)$ Consider the commutative diagram \eqref{brill16} of short exact 
sequences. Let
\[\delta_2'\, :\, \mc{H}^1_D(\mo_X) \, \oplus\, \mc{H}^1_E(\mo_X) \, \longrightarrow\, \mc{H}^1_{D \, \cup \, E}(\mo_X)\]
the natural morphism. Using the Leibniz rule, observe that
\[\delta_2' \, \circ \, (\lrcorner\{D\} \, \oplus\, \lrcorner\{E\})\, :\, \T_X \, \longrightarrow\, \mc{H}^1_{D \, \cup \, E}(\mo_X)\]
coincides with $\lrcorner \{D \, \cup \, E\}$. 
Taking cohomology, this implies $\delta_2 \, \circ \, (\pi_E \, \oplus\, \pi_D)\, =\, \pi_{E \, \cup \, D}$.
As $E \in |\mc{L}_o|$, we have 
$\Phi_E(\mr{Obs}^\pi_{E,\mr{loc}})\, =\, 0$ 
(as there is no obstruction to the deformation of the invertible sheaf
$\mo_X(E)$ along $B$). 
Then, by the commutativity of the lower right hand square of \eqref{brill16}, 
we have for any $t \in T_oB$,
 \[\Phi_D \circ \pi_D \circ 
\mr{KS}(t)=0 \Leftrightarrow (\Phi_D \oplus \Phi_E) \circ 
(\pi_D \oplus \pi_E) \circ \mr{KS}(t)= 0 \oplus 0 
\Leftrightarrow \Phi_{D \cup E} \circ \pi_{D \cup E} \circ 
\mr{KS}(t)\,=\, 0\, .\]
Since $D \cup E$ is effective Lefschetz, 
the restriction of $\Phi_{D \cup E}$
to $\mr{Obs}^\pi_{D \cup E, \mr{loc}}$ is injective.
This implies that for any $t \in T_oB$, we have 
$\Phi_D \circ \pi_D \circ \mr{KS}(t)\, =\, 0$ if and only if 
\[\delta_2((\pi_D \oplus \pi_E) \circ \mr{KS}(t))\, =\, 
 \pi_{D \cup E} \circ \mr{KS}(t)\, =\, 0.\]
By the exactness of the middle row of \eqref{brill16}, this implies 
for any $t \in T_oB$, 
$\Phi_D \circ \pi_D \circ \mr{KS}(t)=0$ if and only if there 
exists $\xi_t \in H^2_T(\mo_X)$ (depending on $t$) such that 
$(\pi_D \oplus \pi_E) \circ \mr{KS}(t)=\delta_1(\xi_t)$.
By the injectivity of the morphism $i$, this is equivalent to 
$\Phi_T(\xi_t)=0$. By the exactness of the first vertical column,
we conclude that for any $t \in T_oB$, 
$\Phi_D \circ \pi_D \circ \mr{KS}(t)=0$ if and only if there 
exists $\xi'_t \in H^1(\mo_{X \backslash T})$ such that 
$\Phi'_D \circ \delta_3(\xi'_t)=\pi_D \circ \mr{KS}(t)$, i.e.,
$\ker(\Phi_D) \cap \Ima(\pi_D\circ \mr{KS})\,=\,\Ima(\Phi'_D 
 \circ \delta_3).$ Note that, $\Phi'_D 
 \circ \delta_3\, =\, \Phi_{(T \subseteq D)}$. 
 This implies that the image of $\mr{Obs}^{\pi}_{D,\mr{loc}}$ in 
 $\mr{coker}(\Phi_{(T \subseteq D)})$ maps isomorphically to 
 $\mr{Obs}^\pi_D$, under the morphism $\eta_{_T}$, i.e., $D$ is 
 $T$-semi-regular. This proves $(2)$.
\end{proof}

\begin{cor}\label{brill09}
 The following are equivalent:
 \begin{enumerate}
\item $D$ is saturated and $T$-semi-regular for every 
proper subvariety $T$ of $D$

\item for any $t \,\in\, T_oB$, 
$D$ deforms as an effective Cartier divisor of the first order
deformation $X_t$ of $X$ along $t$,
if and only if the associated cohomology class $[D]$
deforms as a Hodge class in the sense that $[D] \bigcup 
\mr{KS}(t)\,=\,0$.
\end{enumerate}
\end{cor}

\begin{proof}
 $(1) \Rightarrow (2)$:
 Suppose that $D$ is saturated and $T$-semi-regular for 
 every proper subvariety $T\,\subseteq\, D$. This implies that 
 $\mr{Tors}_D\,=\,0$. Moreover, by Theorem \ref{brnew01}, we have 
 $D$ is effective Lefschetz. This implies that the natural morphism
 \[\mbf{O}_D^{\pi} \,\longrightarrow\, \mr{Obs}^{\pi}_D\]
 is injective, which is a simple reformulation of statement $(2)$.
 
 $(2) \Rightarrow (1)$: Statement $(2)$ implies that the natural
 morphism from $\mbf{O}_D^{\pi}$ to $\mr{Obs}_D^{\pi}$ is 
 injective. This implies that $\mr{Tors}_D\,=\,0$ and $D$ is 
 effective Lefschetz. Using Corollary \ref{cor:sat}, this implies
 $D$ is saturated. Then $(1)$ follows immediately from 
 Theorem \ref{brnew01}. This proves the corollary.
 \end{proof}

\section{Applications}
In \S~\ref{sec:ltop} above, we observed that every divisor $D$ of a
smooth, projective variety $X$ is $T$-semi-regular for some closed
subvariety $T$ of $D$, which means that the obstruction to the
first order deformation of the cohomology class of $D$ as a Hodge class
lies in the cokernel of a natural morphism from $H^1(\mo_{X\backslash T})$
to $H^2_D(\mo_X)$. In this section, we wish to study the variation of 
$T$-semi-regularity in families. We prove the existence of relative 
$T$-semi-regularity in Theorem \ref{brill21}. We use this to 
to re-interpret classical questions in deformation theory, in terms 
of $T$-semi-regularity (see Theorems \ref{brill21}, \ref{th:noe}).

\subsection{Jumping locus of linear system}\label{rev2}

We use notations as in \S~\ref{rev3}.
Let \[\pi\,:\,\mc{X}\,\longrightarrow\, B\] be a flat family of smooth projective 
varieties with a fixed base point $o \,\, \in\, \, B$. Let $\mc{X}_b\, :=\, \pi^{-1}(b)$ 
be the inverse image for any $b \, \in\, B$, and define $X\, :=\, \mc{X}_o\, =\, \pi^{-1}(o)$.
Fix a relative polarization $H$ on $\mc{X}$. Let $\mc{L}$ be an invertible sheaf on $\mc{X}$ 
such that $h^0(\mc{L}_b) \,\ge \,1$ for all $b \, \in\, B$, where $\mc{L}_b\, :=\, \mc{L}|_{\mc{X}_b}$. 
The goal of this subsection is to give a necessary and sufficient condition under which every 
divisor $D \in |\mc{L}_o|$ lifts to an effective Cartier divisor in any
first-order deformation $X_t$ of $X$ along $B$ (Theorem \ref{brill21}).

Denote by $\mc{P}ic_{\mc{X}/B}$ the relative Picard functor 
associated to $\pi$ (see \cite[Chapter $8$, Definition $5$]{bn}).
Recall that $\mc{P}ic_{\mc{X}/B}$ is representable
by a $B$-scheme
(\cite[$\mr{n}^{o}$ $232$, Theorem $3.1$]{bn}), say $\mr{Pic}_{\mc{X}/B}$. Therefore, the invertible sheaf 
$\mc{L}$ on $\mc{X}$ induces a canonical $B$-morphism \[f_{\mc{L}}\, :\, B \, \longrightarrow\, \mr{Pic}_{\mc{X}/B}\]
such that the composition \[B \xrightarrow{f_{\mc{L}}} \mr{Pic}_{\mc{X}/B} \, \longrightarrow\, B\] is the identity map.
Denote by $\mc{D}iv_{\mc{X}/B}$ the relative effective Cartier divisor 
functor (see \cite[p. $212$]{bn}). The functor $\mc{D}iv_{\mc{X}/B}$
is representable by a $B$-scheme, say $\mr{Div}_{\mc{X}/B}$.
Moreover, there is a canonical $B$-morphism:
\[\mf{d}\, :\, \mr{Div}_{\mc{X}/B} \, \longrightarrow\, \mr{Pic}_{\mc{X}/B}\]
which associates to an effective Cartier divisor $D$ of $\mc{X}$, the
invertible sheaf $\mo_{\mc{X}}(D)$ (see \cite[p. $214$]{bn}). Let 
\[\mf{d}_{\mc{L}}\,:\, \mc{D}_{\mc{L}} \,\longrightarrow\, B\]
be the base change of the morphism $\mf{d}$ by the morphism $f_{\mc{L}}$.
By \cite[Theorem $2.1.5$]{huy}, there exists an unique 
locally closed subscheme $W \,\subseteq\, B$ containing $o$
such that \[\mf{d}_{\mc{L}}^{-1}(W) \, \longrightarrow\, W\] is flat and a scheme 
morphism $W' \, \longrightarrow\, B$ factors through $W$ if and only if 
the base change $W' \times_B W \, \longrightarrow\, W'$ is flat 
with every fiber of dimension $m$, where $m\, =\, \dim |\mc{L}_o|$ 
(note that $|\mc{L}_o|\, =\, \mf{d}_{\mc{L}}^{-1}(o)$). This is 
called the \emph{universal property of flattening stratification}.

\begin{defi}
 Let $Y$ be a scheme and $\mc{M}$ an invertible sheaf on $Y$.
 A global section $s$ of $H^0(\mc{M})$ is called \emph{regular}
 if the natural morphism \[s\, :\, \mo_Y \, \longrightarrow\, \mc{M}\]
 induced by multiplication by $s$, is injective.
\end{defi}

The following lemma gives a necessary and sufficient condition for 
regularity of a global section of an invertible sheaf 
over a non-reduced scheme.

\begin{lem}\label{brill17}
Let $A$ be a local Artinian $\mo_{B,o}$-algebra. Define
$\mc{L}_A\, :=\, \mc{L} \otimes_{_{\mo_{B,o}}} A\, ,$
and let
\[
\eta_o\, :\, H^0(\mc{L}_{A}) \, \longrightarrow\, H^0(\mc{L}_o)
\]
be the natural morphism. Take any
$s \,\in \,H^0(\mc{L}_A)$ such that $\eta_o(s) \,\not=\,0$.
Then, $s$ is a regular section.
\end{lem}

\begin{proof}
Let $J$ be the kernel of the morphism from $\mo_{\mc{X}_A}$
to $\mc{L}_A$ induced by multiplication by $s$. 
Denote by $m_A$ the maximal ideal of $A$. Restricting the 
multiplication morphism to the closed fiber, we get a multiplication 
morphism by $\eta_o(s)$. Restricting this morphism further to $J$, 
we get the morphism
\[\eta_o(s)\,:\, \frac{J}{m_A.J} \,\longrightarrow\, \frac{J \otimes \mc{L}_{\mc{A}}}
 {m_A.J \otimes \mc{L}_{\mc{A}}},\]
which is simply multiplication by $\eta_o(s)$. By the definition of
$J$, this is a zero morphism. Since $\eta_o(s)$ is non-zero, 
this means that $J/(m_AJ)$ annihilates the section $\eta_o(s)$.
As $\mc{L}_o$ is an invertible sheaf, this means that $J/(m_AJ)$
must be the zero ideal. By Nakayama lemma, this implies $J\,=\,0$ i.e.,
$s$ is regular. This proves the lemma.
\end{proof}

\begin{thm}\label{brill21}
The following two hold:
\begin{enumerate}
\item Let $\mc{D} \,\subseteq\, \mc{X}$ be a family of 
effective divisors parameterized by $B$ i.e., for 
every $b \in B$, the fiber $\mc{D}_b$ is an effective divisor
in $\mc{X}_b$. Then, there exists a $B$-flat, proper, closed
subvariety $\mc{T}$ of codimension $2$ in $\mc{X}$ such that 
$\mc{D}_b$ is $\mc{T}_b$-semi-regular for every $b \,\in\, B$.

\item $T_oW\, =\, T_oB$ if and only if every $D \, \in\, |\mc{L}_o|$
is saturated and $T$-semi-regular for every proper closed 
subvariety $T \,\subseteq\, D$.
\end{enumerate}
\end{thm}

\begin{proof}
$(1)$: Since $\pi$ is a projective morphism, there exists a relative
very ample line bundle $\mc{H}$ over $\mc{X}$, embedding $\mc{X}$
into a projective space over $B$ and satisfying the conditions
\[H^1(\mo_{\mc{X}_b}(\mc{D}_b) \otimes \mc{H}_b)\,=\, 0\ \ \text{ and }\ \
 H^1(\mc{H}_b)\,=\,0\]
for all $b \,\in\, B$, where $\mc{H}_b\,:=\, \mc{H}|_{\mc{X}_b}$.
The existence of such an invertible sheaf $\mc{H}$ is guaranteed 
by Serre's vanishing theorem. Take a general section, say $H \,\in\, 
H^0(\mc{H})$, and set $\mc{T}\,:=\,H.\mc{D}$. For $H$ general, 
$\mc{T}$ is $B$-flat, proper, closed subvariety of codimension $2$ in 
$\mc{X}$. The proof then follows identically as the proof of Theorem 
\ref{th:br03} after replacing $X$ by $\mc{X}_b$, $E$ by $H_b$ and $D$ by 
$\mc{D}_b$ for all $b \in B$, where $H_b\,:=\,H \cap \mc{X}_b$. 
This proves $(1)$.

$(2)$: 
Let $t\, :\, \Spec(k[\epsilon]/(\epsilon^2)) \, \longrightarrow\, W$ correspond to a 
tangent vector $T_oW$ i.e., the closed point of $\Spec(k[\epsilon]/(\epsilon^2))$
maps to $o$, under $t$. Let
\[\mf{d}_t\, :\, \mc{D}_t \, \longrightarrow\, \Spec(k[\epsilon]/(\epsilon^2))\]
be the base change of $\mf{d}_{\mc{L}}$
under the morphism $t$. By the universal property of flattening 
stratification mentioned above, $t \, \in\, T_oW$ if and only if 
$\mc{D}_t$ is flat over $\Spec(k[\epsilon]/(\epsilon^2))$.
Since $\mc{D}_o\, :=\, \mf{d}_{\mc{L}}^{-1}(o)$ is smooth (projective space), $\mf{d}_t$ is flat
if and only if it is smooth.
This implies that $t \, \in\, T_oW$ if and only if for every 
local Artinian $k[\epsilon]/(\epsilon^2)$-algebra $A$ and a small extension of $A$
\[0 \, \longrightarrow\, I \, \longrightarrow\, A' \xrightarrow{\eta} A \, \longrightarrow\, 0,\]
the natural map
$\mc{D}_t(\Spec(A')) \, \longrightarrow\, \mc{D}_t(\Spec(A)) \, \mbox{ is surjective}.$
Define $$\mc{L}_A\, :=\, \mc{L} \otimes_{_{\mo_{B,o}}} A\ \ \text{ and }\ \
\mc{L}_{A'}\, :=\, \mc{L} \otimes_{_{\mo_{B,o}}} A'\, .$$ 
Recall, elements in $\mc{D}_t(\Spec(A'))$ (respectively, $\mc{D}_t(\Spec(A))$)
is in $1-1$ correspondence with regular global sections of 
$\mc{L}_{A'}$ (respectively, $\mc{L}_A$) (see \cite{stadiv}).
Recall, \cite[Propositions III.$12.5$ and $12.10$]{R1} states that the natural morphism from 
$H^0(\mc{L}_{A'})$ to $H^0(\mc{L}_{A})$ is surjective if and only if the homomorphism
\[\eta_{_t}\, :\, H^0(\mc{L}_t) \, \longrightarrow\, H^0(\mc{L}_o)\]
is surjective. 
 Using Lemma \ref{brill17}, we conclude that
$t \, \in\, T_oW$ if and only if $\eta_{_t}$ is surjective.
By Corollary \ref{brill09}, 
any $D \, \in\, |\mc{L}_o|$ is saturated and $T$-semi-regular 
for every proper, closed subvariety $T \,\subseteq\, D$
if and only if for every $t \, \in\, T_oB$, we have $s_D \,\in \,\Ima(\eta_{_t})$, where 
$s_D \in H^0(\mc{L}_o)$ is a global section corresponding to $D$. 
This proves $(2)$.
\end{proof}

\subsection{Noether-Lefschetz locus}

Denote by $U_d$ the space parameterizing smooth, degree $d$ surfaces in $\p3$,
and let
\[\pi\, :\, \mc{X} \, \longrightarrow\, U_d\] the corresponding universal family.
The \emph{Noether-Lefschetz locus}, denoted by $\NL_d$, parametrizes
smooth, degree $d$ surfaces in $\p3$ with Picard number at least two.
Using the Lefschetz $(1,1)$-theorem, observe that any irreducible component $L$
of $\NL_d$ is locally isomorphic to the Hodge locus $\NL(\gamma)$ for 
some $\gamma \, \in\, H^{1,1}(\mc{X}_o,\,\mb{C}) \, \cap \, H^2(\mc{X}_o,\,\mb{Z})$
and $o \, \in\, L$ (see \cite[\S~ 5.3.3]{v5}).

Denote by $Q_d$ the Hilbert polynomial of a smooth, degree $d$ surface in $\p3$
and $\mr{Hilb}_{P,Q_d}$ the flag Hilbert scheme parameterizing pairs $(C \,\subseteq\, X)$,
where $C$ (respectively, $X$) is of Hilbert polynomial $P$ (respectively, $Q_d$) (see \cite[\S~4.5]{S1}).
Let \[\pr_2\, :\, \mr{Hilb}_{P,Q_d} \, \longrightarrow\,\mr{Hilb}_{Q_d}\] be the natural projection.

\begin{thm}\label{th:noe}
Let $L$ be an irreducible component of $\NL_d$, $o \, \in\, L$ a point 
and $\mc{X}_o$ the corresponding surface. Choose $C \,\subseteq\, \mc{X}_o$ 
such that $\ov{\NL([C])}\, =\, L$ (closure under Zariski topology).
Then, there exists an irreducible component $W$ of $\mr{Hilb}_{P,Q_d}$
such that $\pr_2(W) \,\cong\, L$ if and only if $C$ is 
saturated and $T$-semi-regular for every proper, closed, 
subvariety $T$ of $C$.
\end{thm}

\begin{proof}
This follows directly from Corollary \ref{brill09}.
\end{proof}

\section{An example}

In this section, we produce a family $\pi\,:\,\mc{X} \,\longrightarrow\, B$ 
along with a point $o \, \in\, B$ and an effective Cartier divisor $D \,\subseteq\, \mc{X}_o$,
which does not satisfy the effective Lefschetz property i.e. 
 there exists a first order infinitesimal 
 deformation of $\mc{X}_o$ corresponding to a 
tangent vector $t \in T_oB$ such that $D$ does not deform as an 
effective Cartier divisor but the cohomology class of $D$ deforms 
as a Hodge class (Theorem \ref{brill22}). The example is motivated
by Remark \ref{rem:ref}. In particular, 
the divisor $D$ arises from the limit of the blow-up of 
a fixed smooth, projective surface
along a finite set of points, as the set of points converges
to the base locus of a linear system.

\begin{note}
Let $Y$ be a smooth projective surface and $\mc{L}$ an invertible sheaf on $Y$. 
Denote by $F$ the reduced base locus of $\mc{L}$ (this means the base locus with
reduced scheme structure). Suppose $F \, \not= \, \emptyset$. Fix a point $o \, \in\, F$
such that $o$ is a reduced base point in the sense 
that there exists a global section $s \, \in\, H^0(\mc{L})$
such that its image $s_o$ under the localization morphism
\[H^0(\mc{L}) \, \longrightarrow\, \mc{L}_o\] defines a non-zero element 
in $(m_o/m_o^2)\mc{L}_o$, where $m_o$ is the maximal ideal of $\mo_{Y,o}$.
Denote by $B$ the union of the complement  $Y\backslash F$ and the point 
$o$. Let $D$ be the surface in $Y \times B$ which is the union of 
$\{(b\, ,\, b)\,|\, b \in B\}$ and $(F\, \backslash\, \{o\})
\times\, B.$ Denote by 
\[
\ov{\pi}\, :\, \mc{X} \, \longrightarrow\, Y \times B\]
the blow-up of $Y \times B$ along $D$. Let $E\, \subseteq\, \mc{X}$ be the exceptional divisor. Let
\[\pi\, :\, \mc{X}\, \longrightarrow\, B\]
be the composition of natural morphisms
\[\mc{X} \xrightarrow{\ov{\pi}} Y \times B \xrightarrow{\pr_2} B\, ,\]
and let
\[
\pi\vert_E\, :\, E \,\longrightarrow\,  B\]
be the restriction of $\pi$ to $E$.
Denote by $\mc{X}_b\, :=\, \pi^{-1}(b)$ and $E_b\, :=\, (\pi|_E)^{-1}(b)$
for all $b \, \in\, B$.
\end{note}
 
\begin{lem}\label{br22}
The above morphisms $\pi$ and $\pi\vert_E$ are flat.
\end{lem}

\begin{proof}
Let $\pr_2\, :\, D \, \longrightarrow\, B$ be the natural projection. For every $b \, \in\, B$, 
\[\pr_2^{-1}(b)\, =\, ((F \backslash \{o\}) \times \{b\}) \, \cup \, (\{b\} \times \{b\}).\]
Hence, $\pr_2$ is a flat morphism (to prove this use \cite[Theorem III.$9.9$]{R1}).
By the definition of a blow-up, observe that $\mc{X}$
is flat over $B$, under the morphism $\pi$.
The flatness of $\pi|_E$ follows directly from the 
description of the exceptional divisor of a blow-up 
along a regular subscheme as given in \cite[\S~8.1.2, Theorem $1.19$]{liu}.
This proves the lemma.
\end{proof}
Consider the following short exact sequence of sheaves on $\mc{X}$:
\begin{equation}\label{brill11}
0 \, \longrightarrow\, \mo_{\mc{X}}(-E) \, \longrightarrow\, \mo_{\mc{X}} \, \longrightarrow\, \mo_E \, \longrightarrow\, 0\, .
\end{equation}
We are interested in the invertible sheaf arising from the 
pull-back of $\mc{L}$ under the natural morphism from 
$\mc{X}$ to $Y$, after twisting by the ideal sheaf of the exceptional
divisor. More precisely, consider the composed morphism
$$\pi_1\, :\, \mc{X}\, \xrightarrow{\ov{\pi}}\, Y \times B\, \xrightarrow{\pr_1} Y$$
Define
$$\mc{M}\, :=\, \pi_1^*\mc{L} \, \otimes\, \mo_{\mc{X}}(-E)\, .$$
Tensoring \eqref{brill11} by $\pi_1^*\mc{L}$, we get the following short 
exact sequences:
\begin{align}
& 0 \, \longrightarrow\, \mc{M} \, \longrightarrow\, \pi_1^*\mc{L} \, \longrightarrow\, \pi_1^*\mc{L} \, \otimes\, \mo_E \, \longrightarrow\, 0 \label{brill13}
\end{align}
For any $b \, \in\, B$, let $\mc{M}_b\, :=\, \mc{M}|_{\mc{X}_b}$. 
We then observe:

\begin{thm}\label{brill22}
Let $\mr{KS}\, :\, T_oB \, \longrightarrow\, H^1(\T_{\mc{X}_o})$ be the Kodaira--Spencer map associated to $\pi$.
There exists $t \, \in\, T_oB$ and $C \, \in\, |\mc{M}_o|$ such that $\mr{KS}(t) \, \cup \, [C]\, =\, 0$
but $C$ does not lift to an effective Cartier divisor of the corresponding
first order infinitesimal deformation $\mc{X}_{o,t}$ of $\mc{X}_o$.
\end{thm}

\begin{proof}
Any tangent vector $t \, \in\, T_oB$ corresponds to a ring morphism, 
\[\phi\, :\, \mo_{B,o} \, \longrightarrow\, k(o)[\epsilon]/(\epsilon^2).\] As \eqref{brill13}
is a short exact sequence of $B$-flat $\mo_{\mc{X}}$-modules, 
applying $- \otimes_{\phi} k(o)[\epsilon]/(\epsilon^2)$
preserves exactness:
\[0 \, \longrightarrow\, \mc{M} \otimes_{\phi} \frac{k(o)[\epsilon]}{(\epsilon^2)} \, \longrightarrow\, \pi_1^*\mc{L} \otimes_{\phi} \frac{k(o)[\epsilon]}{(\epsilon^2)} \xrightarrow{\rho_\phi} (\pi_1^*\mc{L} \, \otimes\, \mo_E) \otimes_{\phi} \frac{k(o)[\epsilon]}{(\epsilon^2)}\, \longrightarrow\, 0\, .\]
Given any coherent sheaf $\mc{F}$ on $\mc{X}$, define the morphism 
\[\epsilon\, :\, \left(\mc{F} \otimes_\phi \frac{k(o)[\epsilon]}{(\epsilon^2)}\right) \otimes_{ \frac{k(o)[\epsilon]}{(\epsilon^2)}} k(o) \, \longrightarrow\,\left(\mc{F} \otimes_\phi \frac{k(o)[\epsilon]}{(\epsilon^2)}\right) \otimes_{ \frac{k(o)[\epsilon]}{(\epsilon^2)}} \frac{k(o)[\epsilon]}{(\epsilon^2)},\]
which maps $s \otimes 1$ to $s \, \otimes\,\epsilon$ for $s$ a section of $\mc{F} \otimes_\phi k(o)[\epsilon]/(\epsilon^2)$.
Applying the global section functor to the above short exact sequence, we obtain the following diagram:
$$
\xymatrix{
0 \ar[r] & H^0(\mc{M}_o) \ar[r] \ar[d]^{\epsilon} & H^0(\pi_{1,o}^*\mc{L}) \ar[d]^{\epsilon} \ar[r]^{\rho_o} & H^0(\pi_{1,o}^*\mc{L} \, \otimes\, \mo_{E_o}) \ar[d]^{\epsilon}\\
0 \ar[r] & H^0\left(\mc{M}\otimes_\phi \frac{k(o)[\epsilon]}{(\epsilon^2)}\right) \ar[r] \ar[d]^{r_1} & H^0\left(\pi_{1}^*\mc{L}\otimes_\phi \frac{k(o)[\epsilon]}{(\epsilon^2)}\right)\ar[d]^{r_2} \ar[r]^{\rho_{\phi} \hspace{6mm}} & H^0\left((\pi_{1}^*\mc{L} \, \otimes\,\mo_{E})\otimes_\phi \frac{k(o)[\epsilon]}{(\epsilon^2)}\right) \ar[d]\\
0 \ar[r] & H^0(\mc{M}_o) \ar[r] & H^0(\pi_{1,o}^*\mc{L}) \ar[r]^{\rho_o} & H^0(\pi_{1,o}^*\mc{L} \, \otimes\, \mo_{E_o}) }
$$
where all the rows and columns are exact and the morphisms $\epsilon$ are injective and $\pi_{1,b}: \mc{X}_b \to Y$ is the restriction of $\pi$ to 
$\mc{X}_b$, for any $b \in B$. Using the 
Zariski main theorem and the projection formula, we get 
\begin{align*}
& H^0(\pi_{1,b}^*\mc{L})\, =\, H^0(\pi_{1,b_*}\pi_{1,b}^*\mc{L})\, =\, H^0(\mc{L}) \mbox{ for all } b \, \in\, B \mbox{ and }\\
& H^0(\pi_{1,o}^*\mc{L} \, \otimes\, \mo_{E_o})\, =\, H^0(\pi_{1,o_*}(\pi_{1,o}^*\mc{L} \, \otimes\, \mo_{E_o}))\, =\, H^0(\mc{L} \, \otimes\, \mo_F).
\end{align*}
Since $F$ is the base locus of $\mc{L}$, this means the evaluation map $\rho_o$ is the zero map.
By Grauert's upper semicontinuity theorem \cite[Corollary III.$12.9$]{R1},
the morphism $r_2$ is surjective. If $r_1$ is surjective, then 
the Snake lemma implies that $\rho_\phi$ is the zero map.
Therefore, it suffices to show that $\rho_\phi$ is not the zero map.
Indeed, this would imply that $r_1$ is not surjective.
This directly means that there exists $C \, \in\, |\mc{M}_o|$
not contained in $\Ima r_1$. In other words, 
$C$ does not lift to an effective Cartier divisor of the corresponding
first order 
infinitesimal deformation $\mc{X}_{o,t}$ of $\mc{X}_o$.
This will prove the theorem.
 
Since $\pi_1\, =\, \ov{\pi} \, \circ \, \pr_1$, the projection formula along with Zariski's main theorem implies that 
\[ H^0(\pi_1^*\mc{L} \otimes_\phi k(o)[\epsilon]/(\epsilon^2))\, =\, 
H^0(\pr_1^*\mc{L} \otimes_\phi k(o)[\epsilon]/(\epsilon^2)) \mbox{ and }\]
\[H^0(\pi_1^*\mc{L} \, \otimes\, \mo_E \otimes_\phi k(o)[\epsilon]/(\epsilon^2))\, =\, 
H^0(\pr_1^*\mc{L} \, \otimes\, \mo_{D} \otimes_\phi k(o)[\epsilon]/(\epsilon^2))\,=\,\]
\[\, =\, \bigoplus\limits_{q \, \in\, F} H^0((\pr_1^*\mc{L} \, \otimes\, \mo_{D})_{q \times o}
\otimes_\phi k(p)[\epsilon]/(\epsilon^2))\, .\]
By assumption, there exist $s \,\, \in\, \, H^0(\mc{L})$,
$f_s \,\, \in\, \, m_o\backslash m_o^2$ and $g_s \,\, \in\, \, \mc{L}_o$ (the localization of $\mc{L}$
at $o$) such that 
$s_o\,=\,f_sg_s$ and $g_s \, \not\in\, \, m_o\mc{L}_o$; here $s_o$ is the image of $s$
under the localization morphism $H^0(\mc{L}) \,\longrightarrow\, \mc{L}_o$. 
Since $f_s \,\, \in\, \, m_o\backslash m_o^2$ and $Y$ is a smooth surface, we can choose
a regular sequence $(f_s,\,f_1)$ generating the maximal ideal $m_o$. 
Define, $$\phi\,:\,\mo_{B,o}
\,\longrightarrow\, k(o)[\epsilon]/(\epsilon^2)$$
the ring morphism by
$1 \,\longmapsto\, 1$, $f_s \,\longmapsto\, \epsilon$ and $f_1 \,\longmapsto\, 0$.
Then, $s$ defines a non-zero element
$$s \, \otimes\, 1 \, \in\, H^0(\pr_1^*\mc{L} \otimes_\phi k(o)[\epsilon]/(\epsilon^2))$$
and the image of $s \, \otimes\, 1$ under the natural localization morphism 
 \[\rho(\phi)'_o\,:\,H^0(\pr_1^*\mc{L} \otimes_\phi k(o)[
\epsilon]/(\epsilon^2)) \,\longrightarrow\, H^0(\mc{L}_o \otimes_\phi k(o)[\epsilon]/(\epsilon^2))\] 
is non-zero (use $(\pr_1^* \mc{L})_{(o,o)}\, =\, (\mc{L} \otimes_k \mo_B)_{(o,o)}\, =\, \mc{L}_o \otimes_k \mo_{B,o}$).

Recall, the composition $\Delta \,\hookrightarrow\, Y \times Y \,
\stackrel{\pr_1}{\longrightarrow}\,Y$
is an isomorphism. Hence, \[\pr_1^{\#}\,:\,\mo_{Y,o} 
\,\stackrel{\sim}{\longrightarrow}\,
\mo_{\Delta,o \times o}.\] Since the only irreducible component of $D$ containing
$o \times o$ is $\Delta$, we have 
 \[(\pr_1^*\mc{L} \, \otimes\, \mo_{D})_{o \times o} \otimes_\phi k(o)[\epsilon]/
(\epsilon^2) \,\cong\, \mc{L}_o \otimes_{\pr_1^{\#}} \mo_{\Delta, o \times o}
\otimes_\phi k(o)[\epsilon]/(\epsilon^2)
\,\cong\, \mc{L}_o \otimes_\phi k(o)[\epsilon]/(\epsilon^2).\]
Write the evaluation map $\rho_{\phi}\,=\,\bigoplus\limits_{q \, \in\, F} \rho_{\phi,q}$, 
where $\rho_{\phi,q}$ is the restriction of the evaluation map to $q$. 
Then, $\rho_{\phi,o}$ coincides with the morphism 
\[H^0(\pi_1^*\mc{L} \otimes_\phi k(o)[\epsilon]/(\epsilon^2))\, =\, 
H^0(\pr_1^*\mc{L} \otimes_\phi k(o)[\epsilon]/(\epsilon^2))
\xrightarrow{\rho(\phi)'_o} H^0(\mc{L}_o \otimes_\phi k(o)[\epsilon]/(\epsilon^2)).\]
 Since $\rho(\phi)'_o$ is non-zero, so is $\rho_\phi$.
This completes the proof of the theorem.
\end{proof}

\section*{Acknowledgements}
The comments of the referee helped us to reformulate the results 
of \S~$4$ in the previous version, in terms of the local 
obstruction theory and improve the overall exposition of the article.
We are grateful for his/her feedbacks and suggestions. The first-named author acknowledges 
support of a J. C. Bose Fellowship. The second author is 
currently supported by ERCEA Consolidator Grant $615655$-NMST and also
by the Basque Government through the BERC $2014-2017$ program and by Spanish
Ministry of Economy and Competitiveness MINECO: BCAM Severo Ochoa
excellence accreditation SEV-$2013-0323$.

\end{document}